\documentclass[10pt,a4paper]{amsart}
\usepackage{amsfonts,amssymb}
\usepackage{amsthm,amsmath,amstext,xfrac}
\usepackage{tikz}
\usetikzlibrary{matrix,arrows,calc}
\usepackage{cancel,xspace}
\usepackage{epsfig,fancybox}
\usepackage{graphicx,mathrsfs,ifpdf}
\usepackage{subfig}
\usepackage{setspace}

\newtheorem{theorem}{Theorem}[section]
\newtheorem{lemma}[theorem]{Lemma}

\newtheorem{proposition}[theorem]{Proposition}
\newtheorem{corollary}[theorem]{Corollary}

\theoremstyle{definition}

\numberwithin{equation}{section}
\numberwithin{figure}{section}

\DeclareMathOperator{\Aut}{Aut}
\DeclareMathOperator{\diam}{diam}
\DeclareMathOperator{\dist}{dist}
\DeclareMathOperator{\odd}{odd}

\ifpdf
\usepackage[pdftex]{hyperref}
\else
\newcommand{\href}[2]{#2}
\fi

\newenvironment{tikzgraph}
  {\begin{tikzpicture}
      [vertex/.style={circle, draw=black, fill, inner sep=0mm, minimum size=3pt},edge/.style={semithick},
       subdivision/.style={circle, draw=black, fill=white, inner sep=0mm, minimum size=3pt},edge/.style={semithick}]\begin{scope}}
  {\end{scope}\end{tikzpicture}}

\begin{document} \onehalfspacing
\title[Odd cycle transversals in fullerene graphs]{Odd cycle transversals and independent sets in fullerene graphs}
\thanks{Research supported by CAPES-COFECUB project MA 622/08.}
\author{Luerbio Faria}
\address{Departamento de Matem\'atica, Faculdade de Forma\c c\~ao de Professores, Universidade do Estado do Rio de Janeiro, Brazil}
\email{luerbio@cos.ufrj.br}
\author{Sulamita Klein}
\address{COPPE/Sistemas, Universidade Federal do Rio de Janeiro, Brazil}
\email{sula@cos.ufrj.br}
\author{Mat\v ej Stehl\'ik}
\address{UJF-Grenoble 1 / CNRS / Grenoble-INP, G-SCOP UMR5272 Grenoble, F-38031, France}
\email{matej.stehlik@g-scop.inpg.fr}

\begin{abstract}
A fullerene graph is a cubic bridgeless plane graph with all faces of size $5$ and $6$. We show that that every fullerene graph on $n$ vertices can be made bipartite by deleting at most $\sqrt{12n/5}$ edges, and has an independent set with at least $n/2-\sqrt{3n/5}$ vertices. Both bounds are sharp, and we characterise the extremal graphs. This proves conjectures of Do\v sli\'c and Vuki\v cevi\'c, and of Daugherty. We deduce two further conjectures on the independence number of fullerene graphs, as well as a new upper bound on the smallest eigenvalue of a fullerene graph.
\end{abstract}

\maketitle

\section{Introduction}

A set of edges of a graph is an \emph{odd cycle (edge) transversal} if its removal results in a bipartite graph; the smallest size of an odd cycle transversal of $G$ is denoted by $\tau_{\odd}(G)$. Finding a minimum odd cycle transversal of a graph is equivalent to partitioning the vertex set into two parts, such that the number of edges between the two parts is maximum; this is known as the \emph{max-cut problem} in the literature.

Erd\H os~\cite{Erd65} observed that every graph has an odd cycle transversal containing at most half of its edges, and conjectured that every \emph{triangle-free} graph on $n$ vertices has an odd cycle transversal with at most $\frac1{25}n^2$ edges. Hopkins and Staton~\cite{HopSta82} proved that every triangle-free \emph{cubic} graph on $n$ vertices has an odd cycle transversal with at most $\frac3{10}n$ edges. For triangle-free cubic \emph{planar} graphs, the bound was improved to $\frac7{24}n+\frac76$ by Thomassen~\cite{Tho06}, and subsequently to $\frac9{32}n+\frac9{16}$ by Cui and Wang~\cite{CuWa09}.

A widely studied class of triangle-free cubic planar graphs is the class of \emph{fullerene graphs}: these are cubic bridgeless plane graphs with all faces of size $5$ or $6$. Do\v sli\'c and Vuki\v cevi\'c~\cite[Conjecture 13]{DoVu07} conjectured that every fullerene graph on $n$ vertices has an odd cycle transversal with at most $\sqrt{\frac{12}5n}$ edges, and showed that this bound is attained by fullerene graphs on $60k^2$ vertices, where $k \in \mathbb N$, with the icosahedral automorphism group. Dvo\v r\'ak, Lidick\'y and \v Skrekovski~\cite{DvLiSk12} have recently verified the conjecture asymptotically by proving that $\tau_{\odd}(G) = O(\sqrt n)$. The main result of this paper is a proof of the conjecture of Do\v sli\'c and Vuki\v cevi\'c.

\begin{theorem}\label{thm:main}
If $G$ is a fullerene graph on $n$ vertices, then $\tau_{\odd}(G) \leq \sqrt{\frac{12}5n}$. Equality holds if and only if $n=60k^2$, for some $k \in \mathbb N$, and $\Aut(G) \cong I_h$.
\end{theorem}

The rest of the paper is organised as follows. In Section~\ref{sec:terminology}, we cover the basic notation and terminology. In Section~\ref{sec:T-joins}, we recall the concepts of $T$-joins and $T$-cuts, and establish a bound on the minimum size of a $T$-join in a plane triangulation in terms of the maximum size of a packing of $T$-cuts in an auxiliary plane triangulation. In Section~\ref{sec:patches}, we introduce the notions of patches and moats, and prove bounds on the number of edges in moats. In Section~\ref{sec:proof}, we combine the bounds from the preceding two sections to complete the proof of Theorem~\ref{thm:main}. In Section~\ref{sec:independence}, we deduce a number of conjectures about the independence number of fullerene graphs. Finally, in Section~\ref{sec:eigenvalues}, we compute a new upper bound on the smallest eigenvalue of a fullerene graph.

\section{Notation and terminology}\label{sec:terminology}

Most terminology used in this paper is standard, and may be found in any graph theory textbook. All graphs considered are simple, that is, have no loops and multiple edges. The vertex and edge set of a graph $G$ is denoted by $V(G)$ and $E(G)$, respectively. If $X\subseteq V(G)$ or $X\subseteq E(G)$, we let $G-X$ be the graph obtained from $G$ by removing the elements in $X$, and $G[X]$ the subgraph of $G$ induced by $X$.

A graph is \emph{planar} if it can be drawn in the plane $\mathbb R^2$ so that its vertices are points in $\mathbb R^2$, and its edges are Jordan curves in $\mathbb R^2$ which intersect only at their end-vertices. A planar graph with a fixed embedding is called a \emph{plane graph}. If $G$ is a plane graph, the connected regions of $\mathbb R^2\setminus G$ are the \emph{faces} of $G$. A face of a plane graph $G$ bounded by three edges is a \emph{triangle} of $G$; if every face of $G$ is a triangle, then $G$ is a plane \emph{triangulation}. If $G$ is a plane graph, the \emph{dual graph} $G^*$ is the multigraph with precisely one vertex in each face of $G$, and if $e$ is an edge of $G$, then $G^*$ has an edge $e^*$ crossing $e$ and joining the two vertices of $G^*$ in the two faces of $G$ incident to $e$.

The \emph{distance} $\dist_G(u,v)$ between two vertices $u$ and $v$ in $G$ is the length of a shortest path in $G$ connecting $u$ and $v$. The \emph{open} and \emph{closed $k$-neighbourhood} of a subset $X \subseteq V(G)$ in $G$ are the sets $N_G^k(X)= \{v \in V(G) \mid \dist_G(v,X) = k\}$ and $N_G^k[X]= \{v \in V(G) \mid \dist_G(v,X) \leq k\}$, respectively. The usual open and closed neighbourhood is defined as $N_G(X)=N_G^1(X)$ and $N_G[X]=N_G^1[X]$, respectively. When $X=\{x\}$, we simply write $N_G^k[x]$ and $N_G^k(x)$. The size of the open neighbourhood $N_G(x)$ is the \emph{degree} $d_G(x)$. We let $\delta_G(X)$ be the set of edges of $G$ with exactly one end-vertex in $X$; if $H=G[X]$ we may also write $\delta_G(H)$ for $\delta_G(X)$. A set $C$ of edges is a \emph{cut} of $G$ if $C=\delta_G(X)$, for some $X \subseteq V(G)$. When there is no risk of ambiguity, we may omit the subscripts in the above notation.

An \emph{automorphism} of a graph $G$ is a permutation of the vertices such that adjacency is preserved. The set of all automorphisms of $G$ forms a group, known as the \emph{automorphism group $\Aut(G)$}. The \emph{full icosahedral group} $I_h \cong A_5 \times C_2$ is the group of all symmetries (including reflections) of the regular icosahedron.


\section{\texorpdfstring{$T$-joins and $T$-cuts}{T-joins and T-cuts}}\label{sec:T-joins}

To prove Theorem~\ref{thm:main}, we will consider the dual of a fullerene graph, that is, a plane triangulation $G$ with all vertices of degree $5$ and $6$. We denote by $T$ the set of $5$-vertices of $G$; it follows from Euler's formula that $|T|=12$. The problem is to find a minimal set $J$ of edges such that $G-J$ has no odd-degree vertices. Such a set of edges is known as a $T$-join.

More generally, let $G$ be any graph with a distinguished set $T$ of vertices such that $|T|$ is even. A \emph{$T$-join} of $G$ is a subset $J \subseteq E(G)$ such that $T$ is equal to the set of odd-degree vertices in $G[J]$. The minimum size of a $T$-join of $G$ is denoted by $\tau(G,T)$.

A \emph{$T$-cut} is an edge cut $\delta(X)$ such that $|T\cap X|$ is odd. A \emph{packing} of $T$-cuts is a disjoint collection $\delta(\mathcal F)=\{\delta(X) \mid X \in \mathcal F\}$ of $T$-cuts of $G$; the maximum size of a packing of $T$-cuts is denoted by $\nu(G,T)$. For more information on $T$-joins and $T$-cuts, the reader is referred to~\cite{CCPS98, LoPl86, Sch03}.

Since every $T$-join intersects every $T$-cut, $\nu(G,T) \leq \tau(G,T)$. If $G$ is bipartite, we in fact have equality.

\begin{theorem}[Seymour~\cite{Sey81}]\label{thm:seymour}
For every bipartite graph $G$ and every subset $T \subseteq V(G)$ such that $|T|$ is even, $\tau(G,T) = \nu(G,T)$. \end{theorem}

A family of sets $\mathcal F$ is said to be \emph{laminar} if, for every pair $X,Y \in \mathcal F$, either $X \subseteq Y$, $Y \subseteq X$, or $X \cap Y = \emptyset$. A packing of $T$-cuts $\delta(\mathcal F)$ is said to be laminar if $\mathcal F$ is laminar. A $T$-cut $\delta(X)$ is \emph{inclusion-wise minimal} if no $T$-cut is properly contained in $\delta(X)$. The following proposition can be found in~\cite{FHRV07}.

\begin{proposition}\label{prop:laminar}
For every bipartite graph $G$ and every subset $T \subseteq V(G)$ such that $|T|$ is even, there exists an optimal packing of $T$-cuts in $G$ which is laminar and consists only of inclusion-wise minimal $T$-cuts.
\end{proposition}

Let us remark that the problem of finding a minimum $T$-join is equivalent to the minimum weighted matching problem, which can be solved efficiently using Edmonds' weighted matching algorithm. The problem of finding a maximum packing of $T$-cuts may be considered as the dual problem in the sense of linear programming. Using Theorem~\ref{thm:seymour} and Proposition~\ref{prop:laminar}, it can be shown (see e.g.~\cite{CCPS98}) that there exists an optimal solution of the dual linear program which is half-integral and laminar. Intuitively, this would correspond to a packing of $T$-cuts where the $T$-cuts consist of `half-edges'. This idea was used, in conjunction with the Four Colour Theorem, by Kr\'al' and Voss~\cite{KrVo04} to show that if $G$ is a planar graph and $T\subseteq V(G)$ is the set of odd-degree vertices of $G$, then $\tau(G,T) \leq 2\nu(G,T)$.

Our approach is similar, but rather than dealing with half-edges, we consider a suitable transformation of the graph $G$. Namely, given a plane triangulation $G$, construct the graph $G'$ by subdividing the edges of $G$, that is, replacing the edges of $G$ by internally disjoint paths of length $2$; the graph $G'$ is clearly bipartite. Now construct the graph $G^{\vartriangle}$ from $G'$ by adding three new edges inside every face of $G'$, incident to the three vertices of degree $2$, as shown in Figure~\ref{fig:refinement}. We call $G^{\vartriangle}$ a \emph{refinement} of $G$. Observe that all the vertices in $V(G^{\vartriangle})-V(G)$ have degree $6$ in $G^{\vartriangle}$, so if $T$ is the set of odd-degree vertices of $G$, then $T$ is also the set of odd-degree vertices of $G^{\vartriangle}$.

\begin{figure}
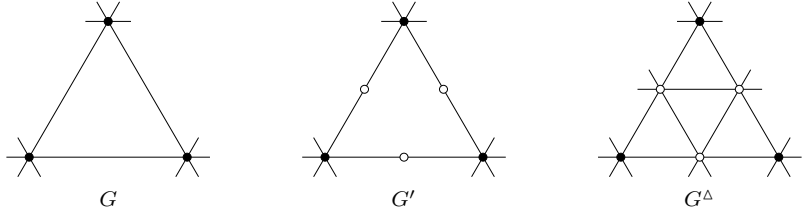

\centering
\null\hfill
\subfloat[$G$]{\label{fig:triangulation}
\begin{tikzgraph}[scale=0.6,ultra thin]
\path (90:2) coordinate (a1);
\path (210:2) coordinate (a2);
\path (330:2) coordinate (a3);

\draw (a1)--(a2)--(a3)--cycle;

\foreach\i in {1,...,3}
{
  \draw (a\i) node[vertex] {};
}
\foreach\i in {0,...,3}
{
  \draw (a1)--+(60*\i:0.5);
}
\foreach\i in {2,...,5}
{
  \draw (a2)--+(60*\i:0.5);
}
\foreach\i in {4,...,7}
{
  \draw (a3)--+(60*\i:0.5);
}
\end{tikzgraph}
}
\hfill
\subfloat[$G'$]{\label{fig:subdivision}
\begin{tikzgraph}[scale=0.6,ultra thin]
\path (90:2) coordinate (a1);
\path (210:2) coordinate (a2);
\path (330:2) coordinate (a3);
\path (150:1) coordinate (s1);
\path (270:1) coordinate (s2);
\path (30:1) coordinate (s3);

\draw (a1)--(a2)--(a3)--cycle;

\foreach\i in {1,...,3}
{
  \draw (a\i) node[vertex] {};
}
\foreach\i in {0,...,3}
{
  \draw (a1)--+(60*\i:0.5);
}
\foreach\i in {2,...,5}
{
  \draw (a2)--+(60*\i:0.5);
}
\foreach\i in {4,...,7}
{
  \draw (a3)--+(60*\i:0.5);
}
\foreach\i in {1,...,3}
{
  \draw (s\i) node[subdivision] {};
}
\end{tikzgraph}
}
\hfill
\subfloat[$G^{\vartriangle}$]{\label{fig:refinement2}
\begin{tikzgraph}[scale=0.6,ultra thin]
\path (90:2) coordinate (a1);
\path (210:2) coordinate (a2);
\path (330:2) coordinate (a3);
\path (150:1) coordinate (s1);
\path (270:1) coordinate (s2);
\path (30:1) coordinate (s3);

\draw (a1)--(a2)--(a3)--cycle;
\draw (s1)--(s2)--(s3)--cycle;

\foreach\i in {1,...,3}
{
  \draw (a\i) node[vertex] {};
}
\foreach\i in {0,...,3}
{
  \draw (a1)--+(60*\i:0.5);
}
\foreach\i in {2,...,5}
{
  \draw (a2)--+(60*\i:0.5);
}
\foreach\i in {4,...,7}
{
  \draw (a3)--+(60*\i:0.5);
}
\foreach\i in {0,...,1}
{
  \draw (s1)--+(120+60*\i:0.5);
}
\foreach\i in {2,...,3}
{
  \draw (s2)--+(120+60*\i:0.5);
}
\foreach\i in {0,...,1}
{
  \draw (s3)--+(60*\i:0.5);
}
\foreach\i in {1,...,3}
{
  \draw (s\i) node[subdivision] {};
}
\end{tikzgraph}
}
\hfill\null
\caption{A face of a triangulation $G$, its subdivision $G'$, and its refinement $G^{\vartriangle}$.}
\label{fig:refinement}
\end{figure}

\begin{lemma}\label{lem:refinement}
For every planar triangulation $G$ and every subset $T \subseteq V(G)$ such that $|T|$ is even, $\tau(G,T) =\frac12\nu(G^{\vartriangle},T)$. Moreover, there exists an optimal laminar packing of inclusion-wise minimal $T$-cuts in $G^{\vartriangle}$.
\end{lemma}

\begin{proof}
Let $G'$ be the subgraph obtained from $G$ by subdividing every edge of $G$. For the first part, it suffices to prove the chain of inequalities
\begin{equation*}\label{eq:tau=nu/2}
\tau(G,T) \leq \tfrac12\tau(G',T) \leq \tfrac12\nu(G',T) \leq \tfrac12\nu(G^{\vartriangle},T) \leq \tau(G,T).
\end{equation*}
Clearly, any $T$-join $J'$ of $G'$ corresponds to a $T$-join $J$ of $G$ such that $|J|=\frac12|J'|$, so $\tau(G,T) \leq \frac12\tau(G',T)$. The second inequality $\tau(G',T)\leq \nu(G',T)$ holds by Theorem~\ref{thm:seymour}. To prove the final inequality $\frac12\nu(G^{\vartriangle},T) \leq \tau(G,T)$, observe that any $T$-join $J$ of $G$ corresponds to a $T$-join $J^{\vartriangle}$ of $G^{\vartriangle}$ such that $|J|=\frac12|J^{\vartriangle}|$. Hence, $\tau(G,T) \geq \frac12\tau(G^{\vartriangle},T) \geq \frac12\nu(G^{\vartriangle},T)$.

It remains to prove the third inequality, namely $\nu(G',T) \leq \nu(G^{\vartriangle},T)$. Let $\mathcal F$ be a laminar family on $V(G')$ minimising $\sum_{X \in \mathcal F}|\delta_{G'}(X)|$, such that $\delta_{G'}(\mathcal F)$ is an optimal packing of inclusion-wise minimal $T$-cuts in $G'$; such a family exists by Proposition~\ref{prop:laminar}. Suppose $\delta_{G^{\vartriangle}}(\mathcal F)$ is not a packing of $T$-cuts in $G^{\vartriangle}$. Then there exist $X_1, X_2 \in \mathcal F$ and an edge $e \in E(G^{\vartriangle}) - E(G')$ such that $e \in \delta_{G^{\vartriangle}}(X_1) \cap \delta_{G^{\vartriangle}}(X_2)$. Therefore $e=x_1x_2$, where $x_1$ and $x_2$ are vertices of $V(G')-V(G)$. By the laminarity of $\mathcal F$, $X_1 \cap X_2 = \emptyset$. Therefore, there exists $i \in \{1,2\}$ such that $x_i$ has a neighbour in $V(G') - X_i$. But then $\delta_{G'}(X_i-\{x_i\})$ is a $T$-cut in $G'$ which is disjoint from all other $T$-cuts of $\delta_{G'}(\mathcal F)$, and $|\delta_{G'}(X_i-\{x_i\})| < |\delta_{G'}(X_i)|$, contradicting the minimality of $\sum_{X \in \mathcal F}|\delta_{G'}(X)|$. Hence, $\delta_{G^{\vartriangle}}(\mathcal F)$ is a laminar packing of $T$-cuts in $G^{\vartriangle}$, so $\nu(G',T) \leq \nu(G^{\vartriangle},T)$.

For the `moreover' part, simply note that the packing $\delta_{G^{\vartriangle}}(\mathcal F)$ from the previous paragraph is an optimal laminar packing of inclusion-wise minimal $T$-cuts in $G^{\vartriangle}$.
\end{proof}

\section{Patches and moats}\label{sec:patches}

Throughout this section, $G$ is a plane triangulation with all vertices of degree $5$ and $6$, and $T$ is the set of $5$-vertices of $G$. A $2$-connected subgraph $H \subset G$ such that all faces of $H$, except the outer face, are triangles, is called a \emph{patch} of $G$. If $C$ is the outer cycle of $H$, and the number of vertices in $T\cap V(H-C)$ is $p$, then $H$ is a \emph{$p$-patch}. We define the \emph{area} $A(H)$ as the number of triangles in $H$. An example of a $3$-patch is shown in Figure~\ref{fig:patch-moat}.

Every $p$-patch with $1 \leq p \leq 5$ satisfies the following isoperimetric inequality, which is an immediate corollary of a more general theorem of Justus~\cite[Theorem~3.3.2]{Jus07}.

\begin{theorem}[Justus~\cite{Jus07}]\label{thm:justus}
Let $G$ be a plane triangulation with all vertices of degree $5$ and $6$, and let $T$ be the set of the $5$-vertices of $G$. If $H \subseteq G$ is a $p$-patch with outer cycle $C$, and $1 \leq p \leq 5$, then
\[
|V(C)| \geq \sqrt{(6-p)A(H)}.
\]
If equality holds, then $p=1$.
\end{theorem}

A \emph{moat} of width $k$ in $G$ surrounding $X \subseteq V(G)$ is a subset of $E(G)$ defined as
\[
\delta_G^k(X)=\bigcup_{i=0}^{k-1}\delta_G\left(N^i[X]\right).
\]
In particular, $\delta_G^1(X)=\delta_G(X)$. If $|T\cap X|=p$, then $\delta_G^k(X)$ is a \emph{$p$-moat} of width $k$. See Figure~\ref{fig:patch-moat} for an example of a $3$-moat of width $2$. If $u \in T$, the $1$-moat $\delta_G^k(\{u\})$ is simply denoted by $\delta_G^k(u)$, and is called a \emph{disk} of radius $k$ centred on $u$. To every moat $\delta_G^k(X)$ corresponds a set of $|\delta_G^k(X)|$ faces, namely the faces incident to at least one edge of $\delta_G^k(X)$. We say that these faces are \emph{spanned} by $\delta_G^k(X)$.

\begin{figure}
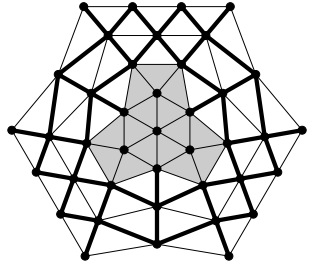

\centering

\begin{tikzgraph}[scale=0.5,ultra thin]

\path (90:1) coordinate (a1);
\path (150:1) coordinate (a2);
\path (210:1) coordinate (a3);
\path (270:1) coordinate (a4);
\path (330:1) coordinate (a5);
\path (30:1) coordinate (a6);

\path (a1)+(50:1) coordinate (b1);
\path (a1)+(130:1) coordinate (b2);
\path (a3)+(170:1) coordinate (b3);
\path (a3)+(250:1) coordinate (b4);
\path (a5)+(290:1) coordinate (b5);
\path (a5)+(370:1) coordinate (b6);

\path (b1)+(50:1) coordinate (c1);
\path (b1)+(130:1) coordinate (c2);
\path (b2)+(130:1) coordinate (c3);
\path (a2)+(150:1) coordinate (c4);
\path (b3)+(170:1) coordinate (c5);
\path (b3)+(250:1) coordinate (c6);
\path (b4)+(250:1) coordinate (c7);
\path (a4)+(270:1) coordinate (c8);
\path (b5)+(290:1) coordinate (c9);
\path (b5)+(10:1) coordinate (c10);
\path (b6)+(370:1) coordinate (c11);
\path (a6)+(30:1) coordinate (c12);

\path (c1)+(50:1) coordinate (d1);
\path (c2)+(50:1) coordinate (d2);
\path (c2)+(130:1) coordinate (d3);
\path (c3)+(130:1) coordinate (d4);
\path (c4)+(150:1) coordinate (d5);
\path (c5)+(170:1) coordinate (d6);
\path (c6)+(170:1) coordinate (d7);
\path (c6)+(250:1) coordinate (d8);
\path (c7)+(250:1) coordinate (d9);
\path (c8)+(270:1) coordinate (d10);
\path (c9)+(290:1) coordinate (d11);
\path (c10)+(290:1) coordinate (d12);
\path (c10)+(10:1) coordinate (d13);
\path (c11)+(10:1) coordinate (d14);
\path (c12)+(30:1) coordinate (d15);

\path[fill=black!20] (b1)--(b2)--(a2)--(b3)--(b4)--(a4)--(b5)--(b6)--(a6)--cycle;

\draw (0,0)--(a1)
      (0,0)--(a2)
      (0,0)--(a3)
      (0,0)--(a4)
      (0,0)--(a5)
      (0,0)--(a6);

\draw (a1)--(b1)
      (a1)--(b2)
      (a3)--(b3)
      (a3)--(b4)
      (a5)--(b5)
      (a5)--(b6);

\draw (a1)--(a2)--(a3)--(a4)--(a5)--(a6)--cycle;
\draw (b1)--(b2)--(a2)--(b3)--(b4)--(a4)--(b5)--(b6)--(a6)--cycle;
\draw (c1)--(c2)--(c3)--(c4)--(c5)--(c6)--(c7)--(c8)--(c9)--(c10)--(c11)--(c12)--cycle;
\draw (d1)--(d2)--(d3)--(d4)--(d5)--(d6)--(d7)--(d8)--(d9)--(d10)--(d11)--(d12)--(d13)--(d14)--(d15)--cycle;
\draw[ultra thick] 
      (b1)--(c1)
      (b1)--(c2)
      (b2)--(c2)
      (b2)--(c3)
      (b2)--(c4)
      (a2)--(c4)
      (b3)--(c4)
      (b3)--(c5)
      (b3)--(c6)
      (b4)--(c6)
      (b4)--(c7)
      (b4)--(c8)
      (a4)--(c8)
      (b5)--(c8)
      (b5)--(c9)
      (b5)--(c10)
      (b6)--(c10)
      (b6)--(c11)
      (b6)--(c12)
      (a6)--(c12)
      (b1)--(c12)
      
      (c1)--(d1)
      (c1)--(d2)
      (c2)--(d2)
      (c2)--(d3)
      (c3)--(d3)
      (c3)--(d4)
      (c3)--(d5)
      (c4)--(d5)
      (c5)--(d5)
      (c5)--(d6)
      (c5)--(d7)
      (c6)--(d7)
      (c6)--(d8)
      (c7)--(d8)
      (c7)--(d9)
      (c7)--(d10)
      (c8)--(d10)
      (c9)--(d10)
      (c9)--(d11)
      (c9)--(d12)
      (c10)--(d12)
      (c10)--(d13)
      (c11)--(d13)
      (c11)--(d14)
      (c11)--(d15)
      (c12)--(d15)
      (c1)--(d15)      
      ;

\draw (0,0) node[vertex] {};

\foreach\i in {1,3,5}
{
  \draw (a\i) node[vertex] {};
}

\foreach\i in {2,4,6}
{
  \draw (a\i) node[vertex] {};
}

\foreach\i in {1,...,6}
{
  \draw (b\i) node[vertex] {};
}

\foreach\i in {1,...,12}
{
  \draw (c\i) node[vertex] {};
}

\foreach\i in {1,...,15}
{
  \draw (d\i) node[vertex] {};
}
\end{tikzgraph}

\caption{A $3$-patch (shaded in grey) surrounded by a $3$-moat of width $2$ (shown by the thick edges).}
\label{fig:patch-moat}
\end{figure}

The number of edges in a disk is easy to determine.

\begin{lemma}
\label{lem:disk}
Let $G$ be a plane triangulation with all vertices of degree $5$ and $6$, and $T$ the set of $5$-vertices of $G$. If $u \in T$, and no edge of $\delta^{k-1}(u)$ is incident to a vertex of $T-\{u\}$, then $\left|\delta_G^k(u)\right| = 5k^2$.
\end{lemma}

\begin{proof}
It is easy to see that $\left|\delta(N^k[u])\right|=5(2k+1)$, so $\left|\delta^k(u)\right|=\sum_{i=0}^{k-1}\left|\delta(N^i[u])\right|=5\sum_{i=0}^{k-1}(2i+1)=5k^2$.
\end{proof}

For more general moats, we can prove the following inequality.
\begin{lemma}
\label{lem:perimeter}
Let $G$ be a plane triangulation with all vertices of degree $5$ and $6$, $T$ the set of $5$-vertices of $G$, and $X \subset V(G)$. If $G[X]$ is a $p$-patch such that $0<p<6$, and no edge of $\delta^{k-1}(X)$ is incident to a vertex of $T$, then
\[
\left|\delta_G^k(X)\right| \geq (6-p)k^2+2k\sqrt{(6-p)A(G[X])}.
\]
If equality holds, then $p=1$.
\end{lemma}

\begin{proof}
Let $C$ be the outer cycle of $G[X]$, and denote by $n$, $m$ and $f$ the number of vertices, edges, and faces (including the outer face) of $G[X]$, respectively. Summing the vertex degrees of $G[X]$ gives $2m=\sum_{v \in V(C)} d_{G[X]}(v)+6(n-|V(C)|)-p$, so
\begin{equation}\label{eq:vertices}
\sum_{v \in V(C)} d_{G[X]}(v)=6|V(C)|+p-6n+2m.
\end{equation}
Summing the face degrees gives $2m=3(f-1)+|V(C)|$, so
\begin{equation}\label{eq:faces}
0=-2|V(C)|+4m-6f+6.
\end{equation}
Adding~\eqref{eq:vertices} and~\eqref{eq:faces},
\begin{equation}\label{eq:boundary}
\sum_{v \in V(C)} d_{G[X]}(v)= 4|V(C)|+p-6(n-m+f-1) =4|V(C)|+p-6,
\end{equation}
where the last equation follows from Euler's formula.

Applying~\eqref{eq:boundary} to the $p$-patch $G[X]$ and the $(12-p)$-patch $G-X$,
\begin{align*}
2|V(C)|+6-p &= \sum_{v \in V(C)}(6-d_{G[X]}(v)) \\
            &= \sum_{v \in N(X)}(6-d_{G-X}(v)) \\
            &= 2|N(X)|-6+p,
\end{align*}
whence $|N(X)|=|V(C)|+6-p$, so by induction,
\begin{equation}\label{eq:N^k(X)}
|N^k(X)|=|V(C)|+(6-p)k.
\end{equation}
By~\eqref{eq:boundary} and~\eqref{eq:N^k(X)}, the number of edges in $\delta(N^k[X])$ is
\begin{align*}
\left|\delta(N^k[X])\right| &= \sum_{v \in N^k(X)}(6-d_{G[X]}(v)) \\
                            &= 2|N^k(X)|+6-p \\
                            &= 2|V(C)|+(6-p)(2k+1),
\end{align*}
\enlargethispage{2\baselineskip}
so the number of edges in $\delta^k(X)$ is
\begin{align*}
\left|\delta^k(X)\right| &= \sum_{i=0}^{k-1}\left|\delta\left(N^i[X]\right)\right|\\
                         &= \sum_{i=0}^{k-1}\left(2|V(C)|+(6-p)(2i+1)\right)\\
                         &= 2k|V(C)|+(6-p)k^2.
\end{align*}
By Theorem~\ref{thm:justus}, $|V(C)| \geq \sqrt{(6-p)A(G[X])}$, with equality only if $p=1$.
\end{proof}

\section{Packing moats in plane triangulations}\label{sec:proof}

When $G$ is a plane triangulation, there exists, by Lemma~\ref{lem:refinement}, an optimal laminar packing $\delta_{G^{\vartriangle}}(\mathcal F)$ of inclusion-wise minimal $T$-cuts in the refinement $G^{\vartriangle}$. We may furthermore assume that the family which gives rise to this packing satisfies $|T\cap X|\leq 5$ for all $X \in \mathcal F$, and minimises $\sum_{X \in \mathcal F} |X|$. We call such a packing a \emph{moat packing}. Let us remark that Kr\'al', Sereni and Stacho~\cite{KrSeSt11+} considered moat packings in bipartite graphs (they used the name \emph{moat solution}). The reason for choosing this name is the following.

For every odd-cardinality subset $U \subset T$, the union of all $T$-cuts in $\delta_{G^{\vartriangle}}(\mathcal F)$ which separate $U$ from $T-U$ is of the form $\delta_{G^{\vartriangle}}^k(X)$, where $U \subseteq X \in \mathcal F$ and $k \in \mathbb N$, i.e., it is a moat of width $k$ surrounding $X$. By the minimality of $\sum_{X \in \mathcal F} |X|$, every $1$-moat in $\delta_{G^{\vartriangle}}(\mathcal F)$ is a disk centred on a vertex $u \in T$, and every vertex of $T$ is the centre of a disk of radius at least $1$. Also by the minimality of $\sum_{X \in \mathcal F} |X|$, if $X \in \mathcal F$ is such that $|X|>1$, then $G[X]$ is $2$-connected. Since every $T$-cut in $\delta_{G^{\vartriangle}}(\mathcal F)$ is inclusion-wise minimal, precisely one face of $G[X]$---the outer face---is not a triangle. Hence, $G[X]$ is a patch, for every $X \in \mathcal F$ such that $|X|>1$.

Therefore, a moat packing of $T$-cuts may be considered as a packing of disks, $3$-moats and $5$-moats. Figure~\ref{fig:tetrahedron} shows an example of such a packing.

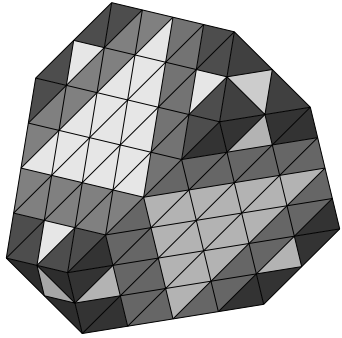
\begin{figure}
\centering
\begin{tikzpicture}[ultra thin,scale=1,line join=bevel,z=-5.5]

\coordinate (A1) at (2,1,1);
\coordinate (A2) at (1,2,1);
\coordinate (A3) at (1,1,2);
\coordinate (B1) at (-2,-1,1);
\coordinate (B2) at (-1,-2,1);
\coordinate (B3) at (-1,-1,2);
\coordinate (C1) at (-2,1,-1);
\coordinate (C2) at (-1,2,-1);
\coordinate (C3) at (-1,1,-2);
\coordinate (D1) at (2,-1,-1);
\coordinate (D2) at (1,-2,-1);
\coordinate (D3) at (1,-1,-2);

\coordinate (A1A2) at (1.5,1.5,1);
\coordinate (A1A3) at (1.5,1,1.5);
\coordinate (A2A3) at (1,1.5,1.5);
\coordinate (B1B2) at (-1.5,-1.5,1);
\coordinate (B1B3) at (-1.5,-1,1.5);
\coordinate (B2B3) at (-1,-1.5,1.5);
\coordinate (C1C2) at (-1.5,1.5,-1);
\coordinate (D1D2) at (1.5,-1.5,-1);
\coordinate (A3B3) at (0,0,2);
\coordinate (AA3B3) at (0.5,0.5,2);
\coordinate (A3BB3) at (-0.5,-0.5,2);
\coordinate (AA1D1) at (2,0.5,0.5);
\coordinate (A1DD1) at (2,-0.5,-0.5);
\coordinate (A1D1) at (2,0,0);
\coordinate (AA1D1) at (2,0.5,0.5);
\coordinate (A1DD1) at (2,-0.5,-0.5);
\coordinate (BB2D2) at (-0.5,-2,0.5);
\coordinate (B2DD2) at (0.5,-2,-0.5);
\coordinate (A2C2) at (0,2,0);
\coordinate (AA2C2) at (0.5,2,0.5);
\coordinate (A2CC2) at (-0.5,2,-0.5);
\coordinate (B1C1) at (-2,0,0);
\coordinate (BB1C1) at (-2,-0.5,0.5);
\coordinate (B1CC1) at (-2,0.5,-0.5);
\coordinate (B2D2) at (0,-2,0);

\draw [draw=none,fill opacity=1,fill=white!80!black] (A1) -- (A2) -- (A3) -- cycle;
\draw [draw=none,fill opacity=1,fill=white!65!black] (B1) -- (B2) -- (B3) -- cycle;
\draw [draw=none,fill opacity=1,fill=white!70!black] (A1) -- (A3) -- (B3) -- (B2) -- (D2) -- (D1) -- cycle;
\draw [draw=none,fill opacity=1,fill=white!90!black] (A2) -- (A3) -- (B3) -- (B1) -- (C1) -- (C2) -- cycle;

\draw [draw=none,fill opacity=1,fill=gray!40!black] (A3) -- (AA3B3) -- (1,0.5,1.5) -- cycle;
\draw [draw=none,fill opacity=1,fill=gray!40!black] (A3) -- (A1A3) -- (1,0.5,1.5) -- cycle;
\draw [draw=none,fill opacity=1,fill=gray!50!black] (A3) -- (A1A3) -- (A2A3) -- cycle;
\draw [draw=none,fill opacity=1,fill=gray!60!black] (A3) -- (0.5,1,1.5) -- (A2A3) -- cycle;
\draw [draw=none,fill opacity=1,fill=gray!60!black] (A3) -- (0.5,1,1.5) -- (AA3B3) -- cycle;

\draw [draw=none,fill opacity=1,fill=gray!40!black] (A1)--(A1A3) -- (1.5,0.5,1) -- cycle;
\draw [draw=none,fill opacity=1,fill=gray!40!black] (A1)--(AA1D1) -- (1.5,0.5,1) -- cycle;
\draw [draw=none,fill opacity=1,fill=gray!50!black] (A1) -- (A1A2) -- (A1A3) -- cycle;

\draw [draw=none,fill opacity=1,fill=gray!60!black] (A2) -- (0.5,1.5,1) -- (A2A3) -- cycle;
\draw [draw=none,fill opacity=1,fill=gray!60!black] (A2) -- (0.5,1.5,1) -- (AA2C2) -- cycle;
\draw [draw=none,fill opacity=1,fill=gray!50!black] (A2) -- (A2A3) -- (A1A2) -- cycle;

\draw [draw=none,fill opacity=1,fill=gray!60!black] (B3) -- (-1,-0.5,1.5) -- (B1B3) -- cycle;
\draw [draw=none,fill opacity=1,fill=gray!60!black] (B3) -- (-1,-0.5,1.5) -- (A3BB3) -- cycle;
\draw [draw=none,fill opacity=1,fill=gray!35!black] (B3) -- (B1B3) -- (B2B3) -- cycle;
\draw [draw=none,fill opacity=1,fill=gray!40!black] (B3) -- (-0.5,-1,1.5) -- (B2B3) -- cycle;
\draw [draw=none,fill opacity=1,fill=gray!40!black] (B3) -- (-0.5,-1,1.5) -- (A3BB3) -- cycle;

\draw [draw=none,fill opacity=1,fill=gray!40!black] (B2) -- (-0.5,-1.5,1) -- (B2B3) -- cycle;
\draw [draw=none,fill opacity=1,fill=gray!40!black] (B2) -- (-0.5,-1.5,1) -- (BB2D2) -- cycle;
\draw [draw=none,fill opacity=1,fill=gray!35!black] (B2) -- (B1B2) -- (B2B3) -- cycle;

\draw [draw=none,fill opacity=1,fill=gray!60!black] (B1) -- (-1.5,-0.5,1) -- (B1B3) -- cycle;
\draw [draw=none,fill opacity=1,fill=gray!60!black] (B1) -- (-1.5,-0.5,1) -- (BB1C1) -- cycle;
\draw [draw=none,fill opacity=1,fill=gray!35!black] (B1) -- (B1B2) -- (B1B3) -- cycle;

\draw [draw=none,fill opacity=1,fill=gray!60!black] (C2) -- (-1,1.5,-0.5) -- (A2CC2) -- cycle;
\draw [draw=none,fill opacity=1,fill=gray!60!black] (C2) -- (-1,1.5,-0.5) -- (C1C2) -- cycle;

\draw [draw=none,fill opacity=1,fill=gray!60!black] (C1) -- (-1.5,1,-0.5) -- (B1CC1) -- cycle;
\draw [draw=none,fill opacity=1,fill=gray!60!black] (C1) -- (-1.5,1,-0.5) -- (C1C2) -- cycle;

\draw [draw=none,fill opacity=1,fill=gray!40!black] (D2) -- (1,-1.5,-0.5) -- (B2DD2) -- cycle;
\draw [draw=none,fill opacity=1,fill=gray!40!black] (D2) -- (1,-1.5,-0.5) -- (D1D2) -- cycle;

\draw [draw=none,fill opacity=1,fill=gray!40!black] (D1) -- (1.5,-1,-0.5) -- (A1DD1) -- cycle;
\draw [draw=none,fill opacity=1,fill=gray!40!black] (D1) -- (1.5,-1,-0.5) -- (D1D2) -- cycle;

\draw [draw=none,fill opacity=1,fill=gray!90!black] (AA2C2)-- (A2C2) -- (A3B3) -- (AA3B3) -- cycle;
\draw [draw=none,fill opacity=1,fill=gray!80!black] (AA3B3)-- (A3B3) -- (A1D1) -- (AA1D1) -- cycle;

\draw [draw=none,fill opacity=1,fill=gray!100!black] (A3BB3)-- (A3B3) -- (B1C1) -- (BB1C1) -- cycle;
\draw [draw=none,fill opacity=1,fill=gray!80!black] (BB2D2)-- (B2D2) -- (A3B3) -- (A3BB3) -- cycle;

\draw [draw=none,fill opacity=1,fill=gray!100!black] (A2CC2)-- (A2C2) -- (B1C1) -- (B1CC1) -- cycle;

\draw [draw=none,fill opacity=1,fill=gray!80!black] (A1DD1)-- (A1D1) -- (B2D2) -- (B2DD2) -- cycle;

\draw (A1) -- (A2) -- (A3) -- cycle;
\draw (B1) -- (B2) -- (B3) -- cycle;
\draw (C1) -- (C2);
\draw (D1) -- (D2);
\draw (A1) -- (D1);
\draw (A2) -- (C2);
\draw (A3) -- (B3);
\draw (B1) -- (C1);
\draw (B2) -- (D2);

\draw (A2C2)--(A3B3)--(A1D1);
\draw (A2A3)--(B1B3);
\draw (A2A3)--(C1C2);
\draw (B1B3)--(C1C2);
\draw (AA2C2)--(AA3B3);
\draw (A2CC2)--(A3BB3);
\draw (AA3B3)--(AA1D1);
\draw (A3BB3)--(A1DD1);
\draw (A3BB3)--(BB2D2);
\draw (AA3B3)--(B2DD2);
\draw (AA1D1)--(BB2D2);
\draw (A1DD1)--(B2DD2);
\draw (BB1C1)--(AA2C2);
\draw (B1CC1)--(A2CC2);
\draw (BB1C1)--(A3BB3);
\draw (B1CC1)--(AA3B3);
\draw (A1A2)--(A1A3)--(A2A3)--cycle;
\draw (B1B2)--(B1B3)--(B2B3)--cycle;
\draw (A1A3)--(B2B3);
\draw (A1A3)--(D1D2);
\draw (B2B3)--(D1D2);
\draw (B1C1)--(A3B3);
\draw (A3B3)--(B2D2);
\draw (A2C2)--(B1C1);
\draw (A1D1)--(B2D2);
\draw (A1)--(B2);
\draw (A3)--(C1);
\draw (A3)--(D2);
\draw (A2)--(B1);
\draw (B3)--(C2);
\draw (B3)--(D1);
\end{tikzpicture}
\caption{A triangulation of the truncated tetrahedron, with a packing of twelve disks and four $3$-moats. The faces spanned by disks are shaded in dark grey, and those spanned by $3$-moats are shaded in light grey. The incidence vectors of this particular packing are $r=1$, $s=1$ and $t=0$.}
\label{fig:tetrahedron}
\end{figure}

We are at last ready to prove Theorem~\ref{thm:main}. To be exact, we first prove the following dual version.

\begin{theorem}\label{thm:T-join}
Let $G$ be a plane triangulation with $f$ faces and all vertices of degree $5$ and $6$. If $T$ is the set of $5$-vertices of $G$, then $\tau(G,T) \leq \sqrt{\frac{12}5f}$, with equality if and only if $f=60k^2$, for some $k \in \mathbb N$, and $\Aut(G) \cong I_h$.
\end{theorem}

\begin{proof}
Let $G^{\vartriangle}$ be the refinement of $G$; so $G^{\vartriangle}$ is a plane triangulation with $4f$ faces and all vertices of degree $5$ and $6$. By Lemma~\ref{lem:refinement}, there exists a moat packing $\delta_{G^{\vartriangle}}(\mathcal F)$. Let $m_1$, $m_3$ and $m_5$ be the number of edges in all disks, $3$-moats, and $5$-moats of $\delta_{G^{\vartriangle}}(\mathcal F)$, respectively. Define the incidence vectors $r, s, t \in \mathbb R^{12}$ as follows: for every $u \in T$, let $r_u$, $s_u$ and $t_u$ be the radius of the disk centred on $u$, the width of the $3$-moat surrounding $u$, and the width of the $5$-moat surrounding $u$, respectively. By the optimality of $\delta_{G^{\vartriangle}}(\mathcal F)$,
\begin{equation}\label{eq:moat_packing}
\tau(G,T) = \tfrac12\nu(G^{\vartriangle},T) = \tfrac12\left\langle r + \tfrac 13 s + \tfrac 15 t, 1\right\rangle,
\end{equation}
where $\langle \cdot,\cdot \rangle$ denotes the inner product.

So to prove the inequality in Theorem~\ref{thm:T-join}, it suffices to find an upper bound on $\left\langle r + \frac 13 s + \frac 15 t, 1\right\rangle$ in terms of $f$. To do so, we compute lower bounds on $m_1$, $m_3$ and $m_5$ in terms of the vectors $r$, $s$ and $t$, and then use the fact that the sum $m_1+m_3+m_5$ cannot exceed $4f$, the number of faces of $G^{\vartriangle}$.

First suppose that $\delta_{G^{\vartriangle}}^{r_u}(u)$ is a disk of $\delta_{G^{\vartriangle}}(\mathcal F)$, for some $u \in T$. Recall that by Lemma~\ref{lem:disk},
\begin{equation}\label{eq:onedisk}
\left|\delta_{G^{\vartriangle}}^{r_u}(u)\right| = 5r_u^2,
\end{equation}
so summing over all disks,
\begin{equation}\label{eq:disk}
m_1 = 5\sum_{u\in T}r_u^2 = 5\|r\|^2,
\end{equation}
where $\|\cdot\|$ denotes the norm.

Now, suppose $\delta_{G^{\vartriangle}}^{s_u}(X)$ is a non-empty $3$-moat of $\delta_{G^{\vartriangle}}(\mathcal F)$, where $u\in T \cap X$ and $|T \cap X| = 3$. The graph $G^{\vartriangle}[X]$ contains $|\delta_{G^{\vartriangle}}^{r_u}(u)|$ triangles spanned by $\delta_{G^{\vartriangle}}^{r_u}(u)$, for every $u \in T \cap X$. All the triangles are pairwise disjoint, so by~\eqref{eq:onedisk} and the Cauchy-Schwarz inequality,
\[
A(G^{\vartriangle}[X])\geq \sum_{u\in T \cap X}\left|\delta_{G^{\vartriangle}}^{r_u}(u)\right| = 5\sum_{u\in T \cap X}r_u^2 \geq \frac53\left(\sum_{u\in T \cap X}r_u\right)^2.
\]
Hence, by Lemma~\ref{lem:perimeter},
\begin{align}
\left|\delta_{G^{\vartriangle}}^{s_u}(X)\right| &\geq 3s_u^2+2s_u\sqrt{3A(G^{\vartriangle}[X])} \notag\\
                                                &\geq 3s_u^2+2\sqrt{5}s_u\sum_{u\in T \cap X}r_u \notag\\
                                                &= \sum_{u\in T \cap X}s_u^2+2\sqrt{5}\sum_{u \in T \cap X}r_us_u\label{eq:one3-moat}.
\end{align}
Summing over all $3$-moats,
\begin{equation}\label{eq:3-moat}
m_3 \geq \|s\|^2+2\sqrt 5\langle r,s \rangle.
\end{equation}

Finally, suppose $\delta_{G^{\vartriangle}}^{t_u}(Y)$ is a non-empty $5$-moat of $\delta_{G^{\vartriangle}}(\mathcal F)$, where $u \in T \cap Y$ and $|T \cap Y| = 5$. By the laminarity of $\delta_{G^{\vartriangle}}(\mathcal F)$, $G^{\vartriangle}[Y]$ contains at most one $3$-moat $\delta_{G^{\vartriangle}}^{s_u}(X)$ of $\delta_{G^{\vartriangle}}(\mathcal F)$, where $X \subset Y$ and $|T \cap X| = 3$. The graph $G^{\vartriangle}[Y]$ contains $|\delta_{G^{\vartriangle}}^{r_u}(u)|$ triangles spanned by $\delta_{G^{\vartriangle}}^{r_u}(u)$, for every $u \in T \cap Y$, as well as at least $\left|\delta_{G^{\vartriangle}}^{s_u}(X)\right|$ triangles spanned by $\delta_{G^{\vartriangle}}^{s_u}(X)$. All the triangles are pairwise disjoint, so by~\eqref{eq:onedisk}, \eqref{eq:one3-moat}, and the Cauchy-Schwarz inequality,
\begin{align*}
A(G^{\vartriangle}[Y]) &\geq \sum_{u\in T \cap Y}\left|\delta_{G^{\vartriangle}}^{r_u}(u)\right|+\left|\delta_{G^{\vartriangle}}^{s_u}(X)\right|\\
                       &\geq 5\sum_{u\in T \cap Y}r_u^2+2\sqrt{5}\sum_{u \in T \cap Y}r_us_u+\sum_{u \in T \cap Y}s_u^2\\
                       &=    5\sum_{u\in T \cap Y}\left(r_u^2+\tfrac1{\sqrt{5}}s_u^2\right)^2\\
                       &\geq \left(\sum_{u \in T \cap Y}r_u+\tfrac{1}{\sqrt 5}\sum_{u \in T \cap Y}s_u \right)^2.
\end{align*}
Hence, by Lemma~\ref{lem:perimeter},
\begin{align*}
\left|\delta_{G^{\vartriangle}}^{t_u}(Y)\right| &\geq t_u^2+2t_u\sqrt{A(G^{\vartriangle}[Y])} \\
                                               &\geq t_u^2+2t_u\sum_{u \in T \cap Y}r_u+\tfrac{2}{\sqrt 5}t_u\sum_{u \in T \cap Y}s_u \\
                                               &=    \tfrac15\sum_{u \in T \cap Y}t_u^2+2\sum_{u \in T \cap Y}r_ut_u+\tfrac{2}{\sqrt 5}\sum_{u \in T \cap Y}s_ut_u.
\end{align*}
Summing over all $5$-moats,
\begin{equation}\label{eq:5-moat}
m_5 \geq \tfrac15\|t\|^2+2\langle r,t \rangle+\tfrac{2}{\sqrt 5}\langle s,t \rangle.
\end{equation}

The graph $G^{\vartriangle}$ has $4f$ triangles, and the disks, $3$-moats and $5$-moats span $m_1$, $m_3$ and $m_5$ triangles of $G^{\vartriangle}$, respectively. These triangles are mutually disjoint, so by~\eqref{eq:disk}, \eqref{eq:3-moat} and~\eqref{eq:5-moat},
\begin{align*}
4f &\geq m_1+m_3+m_5\\
   &\geq 5\|r\|^2 + \|s\|^2 + 2\sqrt 5\langle r,s \rangle + \tfrac15\|t\|^2 + 2\langle r,t \rangle + \tfrac{2}{\sqrt 5}\langle s,t \rangle\\
   &= \left\|\sqrt 5r + s + \tfrac 1{\sqrt 5}t\right\|^2.
\end{align*}
Hence, by the Cauchy-Schwarz inequality and~\eqref{eq:moat_packing},
\begin{align}
\sqrt{\frac{12f}5} & \geq \sqrt{3}\left\|r + \tfrac 1{\sqrt 5}s + \tfrac 15t\right\| \notag\\
                   & \geq \tfrac12\left\langle r + \tfrac 1{\sqrt 5} s + \tfrac 15 t,1\right\rangle \label{eq:cauchy-schwarz}\\
                   & \geq \tau(G,T). \notag
\end{align}

To prove the last part of Theorem~\ref{thm:T-join}, suppose that $\tau(G,T)=\sqrt{\frac{12}5f}$. Equality must hold in~\eqref{eq:3-moat} and~\eqref{eq:5-moat}, so by Lemma~\ref{lem:perimeter}, $s = t = 0$. Furthermore, equality must hold in~\eqref{eq:cauchy-schwarz}, so $r_u=r_v$ for every $u, v \in T$. Therefore $4f=5\cdot12 r_u^2$, so $f=15r_u^2$. Since $f$ is even, it follows that $r_u = 2k$, and therefore $f=60k^2$, for some $k \in \mathbb N$. To see that $\Aut(G) \cong I_h$, note that the graph $G$ may be constructed from the dodecahedron by inserting into each face a $1$-patch of the form $G[N^k[u]]$.

Conversely, if $G$ is a plane triangulation with $f=60k^2$ faces, all vertices of degree $5$ and $6$, and $\Aut(G) \cong I_h$, then $G$ may be constructed from the dodecahedron by inserting into each face a $1$-patch of the form $G[N^k[u]]$. Hence $\dist(u,v) \geq 2k$, for every pair of distinct vertices in $T$, so $\tau(G,T) \geq 12k=\sqrt{\frac{12}5f}$.
\end{proof}

By applying Theorem~\ref{thm:T-join} to the dual graph, we obtain a proof of Theorem~\ref{thm:main}.

\begin{proof}[Proof of Theorem~\ref{thm:main}]
Let $G$ be a fullerene graph on $n$ vertices. The dual graph $G^*$ is a plane triangulation with $n$ faces and all vertices of degree $5$ and $6$. Let $T$ be the set of vertices of degree $5$, $J^*$ a minimum $T$-join of $G^*$, and $J$ the set of edges of $G$ which correspond to $J^*$. Since $G^*-J^*$ has no odd-degree vertices, $G-J=(G^*-J^*)^*$ is bipartite, and by Theorem~\ref{thm:T-join}, $|J|=|J^*| \leq \sqrt{\frac{12}5n}$, with equality if and only if $n=60k^2$, for some $k \in \mathbb N$ and $\Aut(G) \cong I_h$.
\end{proof}

\section{Independent sets in fullerene graphs}\label{sec:independence}

Recall that a set $X \subseteq V(G)$ is \emph{independent} if the graph $G[X]$ has no edges; the maximum size of an independent set in $G$ is the \emph{independence number $\alpha(G)$}. By the Four Colour Theorem, every planar graph on $n$ vertices has an independent set with at least $\frac 14n$ vertices, and by Brooks' Theorem, every triangle-free, cubic graph on $n$ vertices has an independent set with at least $\frac 13n$ vertices. For triangle-free, cubic, planar graphs, the bound can be improved a little further.

\begin{theorem}[Heckman and Thomas~\cite{HeTh06}]\label{thm:heckman-thomas}
If $G$ is a triangle-free cubic planar graph on $n$ vertices, then $\alpha(G) \geq \frac38n$.
\end{theorem}

Daugherty~\cite[Conjecture~5.5.2]{Dau09} conjectured that every fullerene graph on $n$ vertices has an independent set with at least $\frac12n-\sqrt{\frac35n}$ vertices. He also conjectured~\cite[Conjecture~5.5.1]{Dau09} that every fullerene graph attaining this bound has the icosahedral automorphism group and $60k^2$ vertices, for some $k \in \mathbb N$. Andova et al.~\cite{ADKLS12+} recently proved that every fullerene graph on $n$ vertices has an independent set with at least $\frac12n-78.58\sqrt{n}$ vertices. Theorem~\ref{thm:main} immediately implies both conjectures of Daugherty.

\begin{corollary}
If $G$ is a fullerene graph on $n$ vertices, then $\alpha(G) \geq \frac 12n-\sqrt{\frac35n}$, with equality if and only if $n=60k^2$, for some $k \in \mathbb N$, and $\Aut(G) \cong I_h$.
\label{cor:daugherty}
\end{corollary}

\begin{proof}
Every graph $G$ contains an odd cycle vertex transversal $U$ such that $|U| \leq \tau_{\odd}(G)$, so $\alpha(G) \geq \alpha(G-U) \geq \frac12n-\frac12\tau_{\odd}(G)$. Therefore, by Theorem~\ref{thm:main}, $\alpha(G) \geq \frac12n-\sqrt{\frac35n}$, for every fullerene graph $G$. When $J^*$ is a minimum $T$-join of $G^*$, every face of $G^*$ is incident to at most one edge of $J^*$. This means that the set $J \subset E(G)$ corresponding to $J^*$ is a matching of $G$. Therefore, by Theorem~\ref{thm:main}, equality holds if and only if $n=60k^2$, for some $k \in \mathbb N$, and $\Aut(G) \cong I_h$.
\end{proof}

The \emph{diameter} of a graph $G$, denoted $\diam(G)$, is defined as the maximum distance over all pairs of vertices $u, v$ of $G$. The diameter of fullerene graphs satisfies the following upper bound.

\begin{theorem}[Andova et al.~\cite{ADKLS12+}]\label{thm:diameter}
If $G$ is a fullerene graph on $n$ vertices, then $\diam(G) \leq \frac15n+1$.
\end{theorem}

Corollary~\ref{cor:daugherty}, in conjunction with Theorems~\ref{thm:heckman-thomas} and~\ref{thm:diameter}, allows us to prove a conjecture of Graffiti~\cite[Conjecture 912]{FRFHC01}. Let us remark that the conjecture was proved for fullerene graphs on at least 617 502 vertices by Andova et al.~\cite{ADKLS12+}.

\begin{corollary}
If $G$ is a fullerene graph, then $\alpha(G) \geq 2(\diam(G)-1)$.
\end{corollary}

\begin{proof}
Let $G$ be a fullerene graph on $n$ vertices. It is easy to check that $\left\lceil \frac38n \right\rceil \geq \left\lfloor \frac25n \right\rfloor$ if $n < 40$, and $\left\lceil \frac12n-\sqrt{\frac35n}\right\rceil \geq \left\lfloor \frac25n \right\rfloor$ if $n \geq 36$. In the former case, we apply Theorems~\ref{thm:heckman-thomas} and~\ref{thm:diameter}, and in the latter case, we apply Corollary~\ref{cor:daugherty} and Theorem~\ref{thm:diameter}, to show that $\alpha(G) \geq 2(\diam(G)-1)$.
\end{proof}

Motivated by H\"uckel theory from chemistry, Daugherty, Myrvold and Fowler~\cite{DaMyFo07} (see also~\cite{Dau09}) defined the \emph{closed-shell independence number} $\alpha^-(G)$ of a fullerene graph $G$ as the maximum size of an independent set $A$ of $G$ with the property that exactly half of the eigenvalues of $G-A$ are positive. Recall that an \emph{eigenvalue} of a graph $G$ is an eigenvalue of its \emph{adjacency matrix}, the square $n \times n$ matrix $(a_{uv})$ where $a_{uv}=1$ if $uv \in E(G)$, and $a_{uv}=0$ otherwise.

\begin{theorem}[Daugherty, Myrvold and Fowler~\cite{DaMyFo07}]\label{thm:closed-shell}
If $G$ is a fullerene graph, then $\alpha^-(G)\leq \frac38n+\frac32$.
\end{theorem}

Daugherty, Myrvold and Fowler~\cite{DaMyFo07} (see also~\cite[Conjecture~7.7.1]{Dau09}) conjectured that the equality $\alpha^-(G)=\alpha(G)$ holds only when $G$ is isomorphic to one of the three fullerene graphs in Figure~\ref{fig:shell}, and verified the conjecture for all fullerene graphs on $n \leq 100$ vertices. Corollary~\ref{cor:daugherty} and Theorem~\ref{thm:closed-shell} imply the conjecture for all fullerene graphs on $n>60$ vertices, so the conjecture is now proved completely.

\begin{corollary}
\label{cor:shell}
A fullerene graph $G$ satisfies $\alpha^-(G)=\alpha(G)$ if and only if $G$ is one of the graphs in Figure~\ref{fig:shell}.
\end{corollary}

\begin{proof}
Let $G$ be a fullerene graph on $n$ vertices. The conjecture was verified for $n \leq 100$ in~\cite{Dau09}, so it suffices to consider the case $n > 100$. Since $\left\lfloor \frac38n+\frac32 \right\rfloor < \left\lceil 12n-\sqrt{\frac35n} \right\rceil$ for $n>60$, it follows by Corollary~\ref{cor:daugherty} and Theorem~\ref{thm:closed-shell} that $\alpha^-(G)<\alpha(G)$ for $n>60$. 
\end{proof}

\begin{figure}
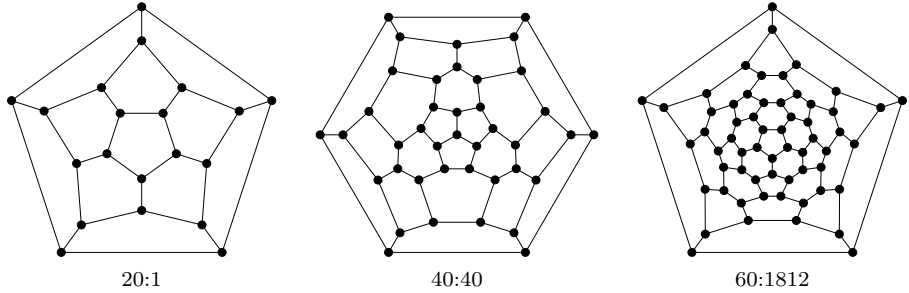

\centering
\null\hfill
\subfloat[$20$:$1$]{\label{fig:C20}
\begin{tikzgraph}[scale=0.6,ultra thin]
\path (54:0.8) coordinate (a1);
\path (126:0.8) coordinate (a2);
\path (198:0.8) coordinate (a3);
\path (270:0.8) coordinate (a4);
\path (342:0.8) coordinate (a5);

\path (90:2.25) coordinate (b1);
\path (126:1.5) coordinate (b2);
\path (162:2.25) coordinate (b3);
\path (198:1.5) coordinate (b4);
\path (234:2.25) coordinate (b5);
\path (270:1.5) coordinate (b6);
\path (306:2.25) coordinate (b7);
\path (342:1.5) coordinate (b8);
\path (18:2.25) coordinate (b9);
\path (54:1.5) coordinate (b10);

\path (90:3) coordinate (c1);
\path (162:3) coordinate (c2);
\path (234:3) coordinate (c3);
\path (306:3) coordinate (c4);
\path (18:3) coordinate (c5);

\draw (a1)--(a2)--(a3)--(a4)--(a5)--cycle
      (b1)--(b2)--(b3)--(b4)--(b5)--(b6)--(b7)--(b8)--(b9)--(b10)--cycle
      (c1)--(c2)--(c3)--(c4)--(c5)--cycle
      (a1)--(b10) (a2)--(b2) (a3)--(b4) (a4)--(b6) (a5)--(b8)
      (b1)--(c1) (b3)--(c2) (b5)--(c3) (b7)--(c4) (b9)--(c5);

\foreach\i in {1,...,5}
{
  \draw (a\i) node[vertex] {};
}

\foreach\i in {1,...,10}
{
  \draw (b\i) node[vertex] {};
}

\foreach\i in {1,...,5}
{
  \draw (c\i) node[vertex] {};
}
\end{tikzgraph}
}
\hfill
\subfloat[$40$:$40$]{\label{fig:C40}
\begin{tikzgraph}[scale=0.6,ultra thin]
\path (90:0.8-0.3) coordinate (a1);
\path (130:0.8) coordinate (a2);
\path (170:0.8) coordinate (a3);
\path (210:0.8-0.3) coordinate (a4);
\path (250:0.8) coordinate (a5);
\path (290:0.8) coordinate (a6);
\path (330:0.8-0.3) coordinate (a7);
\path (10:0.8) coordinate (a8);
\path (50:0.8) coordinate (a9);

\path (90:2-0.5) coordinate (b1);
\path (110:2-0.7) coordinate (b2);
\path (190:2-0.7) coordinate (b3);
\path (210:2-0.5) coordinate (b4);
\path (230:2-0.7) coordinate (b5);
\path (310:2-0.7) coordinate (b6);
\path (330:2-0.5) coordinate (b7);
\path (350:2-0.7) coordinate (b8);
\path (70:2-0.7) coordinate (b9);

\path (90:2) coordinate (c1);
\path (110+25:2) coordinate (c2);
\path (190-25:2) coordinate (c3);
\path (210:2) coordinate (c4);
\path (230+25:2) coordinate (c5);
\path (310-25:2) coordinate (c6);
\path (330:2) coordinate (c7);
\path (350+25:2) coordinate (c8);
\path (70-25:2) coordinate (c9);

\path (100+20:2.5) coordinate (d1);
\path (200-20:2.5) coordinate (d2);
\path (220+20:2.5) coordinate (d3);
\path (320-20:2.5) coordinate (d4);
\path (340+20:2.5) coordinate (d5);
\path (80-20:2.5) coordinate (d6);

\path (120:3) coordinate (e1);
\path (180:3) coordinate (e2);
\path (240:3) coordinate (e3);
\path (300:3) coordinate (e4);
\path (0:3) coordinate (e5);
\path (60:3) coordinate (e6);

\draw (0,0)--(a1) (0,0)--(a4) (0,0)--(a7)
      (a1)--(a2)--(a3)--(a4)--(a5)--(a6)--(a7)--(a8)--(a9)--cycle
      (a9)--(b9)--(b1)--(b2)--(a2)
      (a3)--(b3)--(b4)--(b5)--(a5)
      (a6)--(b6)--(b7)--(b8)--(a8)
      (b1)--(c1) (b2)--(c2) (b3)--(c3) (b4)--(c4) (b5)--(c5) (b6)--(c6) (b7)--(c7) (b8)--(c8) (b9)--(c9)
      (c1)--(d1)--(c2)--(c3)--(d2)--(c4)--(d3)--(c5)--(c6)--(d4)--(c7)--(d5)--(c8)--(c9)--(d6)--cycle
      (d1)--(e1) (d2)--(e2) (d3)--(e3) (d4)--(e4) (d5)--(e5) (d6)--(e6)
      (e1)--(e2)--(e3)--(e4)--(e5)--(e6)--cycle;

\draw (0,0) node[vertex] {};

\foreach\i in {1,...,9}
{
  \draw (a\i) node[vertex] {};
}

\foreach\i in {1,...,9}
{
  \draw (b\i) node[vertex] {};
}

\foreach\i in {1,...,9}
{
  \draw (c\i) node[vertex] {};
}

\foreach\i in {1,...,6}
{
  \draw (d\i) node[vertex] {};
}

\foreach\i in {1,...,6}
{
  \draw (e\i) node[vertex] {};
}

\end{tikzgraph}
}
\hfill
\subfloat[$60$:$1812$]{\label{fig:C60}

\begin{tikzgraph}[scale=0.6,ultra thin]

\path (126:0.35) coordinate (a1);
\path (198:0.35) coordinate (a2);
\path (270:0.35) coordinate (a3);
\path (342:0.35) coordinate (a4);
\path (54:0.35) coordinate (a5);

\path (102:1-0.1) coordinate (b1);
\path (126:1-0.3) coordinate (b2);
\path (150:1-0.1) coordinate (b3);
\path (174:1-0.1) coordinate (b4);
\path (198:1-0.3) coordinate (b5);
\path (222:1-0.1) coordinate (b6);
\path (246:1-0.1) coordinate (b7);
\path (270:1-0.3) coordinate (b8);
\path (294:1-0.1) coordinate (b9);
\path (318:1-0.1) coordinate (b10);
\path (342:1-0.3) coordinate (b11);
\path (6:1-0.1) coordinate (b12);
\path (30:1-0.1) coordinate (b13);
\path (54:1-0.3) coordinate (b14);
\path (78:1-0.1) coordinate (b15);

\path (99:1.5) coordinate (c1);
\path (117:1.5-0.3) coordinate (c2);
\path (135:1.5-0.3) coordinate (c3);
\path (153:1.5) coordinate (c4);
\path (171:1.5) coordinate (c5);
\path (189:1.5-0.3) coordinate (c6);
\path (207:1.5-0.3) coordinate (c7);
\path (225:1.5) coordinate (c8);
\path (243:1.5) coordinate (c9);
\path (261:1.5-0.3) coordinate (c10);
\path (279:1.5-0.3) coordinate (c11);
\path (297:1.5) coordinate (c12);
\path (315:1.5) coordinate (c13);
\path (333:1.5-0.3) coordinate (c14);
\path (351:1.5-0.3) coordinate (c15);
\path (9:1.5) coordinate (c16);
\path (27:1.5) coordinate (c17);
\path (45:1.5-0.3) coordinate (c18);
\path (63:1.5-0.3) coordinate (c19);
\path (81:1.5) coordinate (c20);

\path (90:2.5) coordinate (d1);
\path (114-5:2-0.2) coordinate (d2);
\path (138+5:2-0.2) coordinate (d3);
\path (162:2.5) coordinate (d4);
\path (186-5:2-0.2) coordinate (d5);
\path (210+5:2-0.2) coordinate (d6);
\path (234:2.5) coordinate (d7);
\path (258-5:2-0.2) coordinate (d8);
\path (282+5:2-0.2) coordinate (d9);
\path (306:2.5) coordinate (d10);
\path (330-5:2-0.2) coordinate (d11);
\path (354+5:2-0.2) coordinate (d12);
\path (18:2.5) coordinate (d13);
\path (44-5:2-0.2) coordinate (d14);
\path (68+5:2-0.2) coordinate (d15);

\path (90:3) coordinate (e1);
\path (162:3) coordinate (e2);
\path (234:3) coordinate (e3);
\path (306:3) coordinate (e4);
\path (18:3) coordinate (e5);

\draw (a1)--(a2)--(a3)--(a4)--(a5)--cycle
      (b1)--(b2)--(b3)--(b4)--(b5)--(b6)--(b7)--(b8)--(b9)--(b10)--(b11)--(b12)--(b13)--(b14)--(b15)--cycle
      (a1)--(b2) (a2)--(b5) (a3)--(b8) (a4)--(b11) (a5)--(b14)
      (c1)--(c2)--(c3)--(c4)--(c5)--(c6)--(c7)--(c8)--(c9)--(c10)--(c11)--(c12)--(c13)--(c14)--(c15)--(c16)--(c17)--(c18)--(c19)--(c20)--cycle
      (b1)--(c2) (b3)--(c3) (b4)--(c6) (b6)--(c7) (b7)--(c10) (b9)--(c11) (b10)--(c14) (b12)--(c15) (b13)--(c18) (b15)--(c19)
      (d1)--(d2)--(d3)--(d4)--(d5)--(d6)--(d7)--(d8)--(d9)--(d10)--(d11)--(d12)--(d13)--(d14)--(d15)--cycle
      (c1)--(d2) (c4)--(d3) (c5)--(d5) (c8)--(d6) (c9)--(d8) (c12)--(d9) (c13)--(d11) (c16)--(d12) (c17)--(d14) (c20)--(d15)
      (e1)--(e2)--(e3)--(e4)--(e5)--cycle
      (d1)--(e1) (d4)--(e2) (d7)--(e3) (d10)--(e4) (d13)--(e5)
;

\foreach\i in {1,...,5}
{
  \draw (a\i) node[vertex] {};
}

\foreach\i in {1,...,15}
{
  \draw (b\i) node[vertex] {};
}

\foreach\i in {1,...,20}
{
  \draw (c\i) node[vertex] {};
}

\foreach\i in {1,...,15}
{
  \draw (d\i) node[vertex] {};
}

\foreach\i in {1,...,5}
{
  \draw (e\i) node[vertex] {};
}

\end{tikzgraph}
}
\hfill\null
\caption{The three graphs in Corollary~\ref{cor:shell}, with the nomenclature of~\cite{FoMa95}. The graph $20$:$1$ is the dodecahedral graph, $40$:$40$ is the unique fullerene graph on $40$ vertices with the tetrahedral automorphism group $T_d$, and $60$:$1812$ is the buckminsterfullerene graph.}
\label{fig:shell}
\end{figure}

\section{Smallest eigenvalues of fullerene graphs}\label{sec:eigenvalues}

As the final application of Theorem~\ref{thm:main}, we compute an upper bound on the smallest eigenvalue of a fullerene graph $G$. Recall that the \emph{Laplacian} of a graph with adjacency matrix $(a_{uv})$ is the $n \times n$ matrix $(c_{uv})$, where $c_{uv}=d(u)$ if $u=v$, and $c_{uv}=-a_{uv}$ if $u \neq v$. A \emph{Laplacian eigenvalue} of a graph is an eigenvalue of its Laplacian. The smallest eigenvalue and the largest Laplacian eigenvalue of $G$ are denoted by $\lambda_n(G)$ and $\mu_n(G)$, respectively.

The maximum size of a cut in a graph can be bounded in terms of its largest Laplacian eigenvalue. The following is a corollary of a more general theorem of Mohar and Poljak~\cite{MohPol90}.

\begin{theorem}[Mohar and Poljak~\cite{MohPol90}]\label{thm:mohar-poljak}
If $G$ is a graph on $n$ vertices, then $|\delta(X)| \leq \frac 14n\mu_n(G)$, for every $X \subseteq V(G)$.
\end{theorem}

Andova et al.~\cite{ADKLS12+} have recently used Theorem~\ref{thm:mohar-poljak} to show that $\lambda_n(G) \leq -3+\frac{157.16}{\sqrt{n}}$ for every fullerene graph $G$. Their bound can be improved by applying Corollary~\ref{cor:daugherty}.

\begin{corollary}\label{cor:eigenvalue}
If $G$ is a fullerene graph on $n$ vertices, then $\lambda_n(G) \leq -3+8\sqrt{\frac3{5n}}$.
\end{corollary}

\begin{proof}
Since $G$ is $3$-regular, the smallest eigenvalue of $G$ is $\lambda_n(G)=3-\mu_n(G)$, and there exists a cut $\delta(X)$ such that $|\delta(X)| \geq \frac32n-\tau_{\odd}(G)$. Therefore, by Theorem~\ref{thm:mohar-poljak}, $\lambda_n(G) \leq -3+\frac 4n\tau_{\odd}(G)$, so by Theorem~\ref{thm:main}, $\lambda_n(G) \leq -3+8\sqrt{\frac3{5n}}$.
\end{proof}

Fowler, Hansen and Stevanovi\'c~\cite{FoHaSt03} showed that the smallest eigenvalue of the truncated icosahedron (see Figure~\ref{fig:C60}) is equal to $-\phi^2$, where $\phi$ is the golden ratio $\frac{1+\sqrt 5}2$, and conjectured that, among all fullerene graphs on at least $60$ vertices, the truncated icosahedron has the maximum smallest eigenvalue. By Corollary~\ref{cor:eigenvalue}, any fullerene graph on at least $264$ vertices satisfies the conjecture.

\section*{Acknowledgements}
The authors would like to thank Andr\'as Seb\H o for teaching them about $T$-joins and $T$-cuts, to Louis Esperet for reading an earlier draft of this paper, and to Dragan Stevanovi\'c for pointing out a gap in the proof of Theorem~\ref{thm:T-join}.

\bibliographystyle{plain}
\bibliography{FaKlSt12}
\enlargethispage{2\baselineskip}

\end{document}